\documentclass[pdftex,a4paper,12pt]{amsart}

\usepackage{geometry}
\renewcommand{\arraystretch}{1.1}

\usepackage[cp1250]{inputenc}

\usepackage{alltt}
\usepackage{euler}
\usepackage{amsfonts}
\usepackage{amscd}
\usepackage{graphicx}
\usepackage{lpic}
\usepackage{longtable}
\usepackage{url}
\usepackage{import}
\usepackage{amssymb, amsthm, amsmath}
\usepackage[sc]{mathpazo}
\usepackage{enumitem, textcomp, setspace}
\usepackage{nicefrac}

\usepackage{tikz}
\usepackage{tikz-cd}
\usetikzlibrary{bbox}
\usetikzlibrary{intersections}
\usepackage{pgfplots}
\usepgfplotslibrary{external}

\allowdisplaybreaks

\numberwithin{equation}{section}

\theoremstyle{plain}
\newtheorem{theorem}{Theorem}[section]

\newtheorem{corollary}[theorem]{Corollary}

\theoremstyle{definition}

\usepackage[colorlinks]{hyperref}

\usepackage[all]{hypcap}

\definecolor{internalLink}{rgb}{0,0,0.5}
\definecolor{citeLink}{rgb}{0,0.5,0}
\definecolor{urlLink}{rgb}{0,0.5,0.5}

\hypersetup{
	linkcolor=internalLink,
	citecolor=citeLink,
	urlcolor=urlLink
}

\title[Upper and lower bound on delta-crossing number and tabulation of knots]{Upper and lower bound on delta-crossing number\\ and tabulation of knots up to four delta-crossings}

\author{Micha{\l} Jab{\l}onowski}

\address{Institute of Mathematics, Faculty of Mathematics, Physics and Informatics,\newline University of Gda\'nsk, 80-308 Gda\'nsk, Poland}



\email{michal.jablonowski@gmail.com}

\date{\today}

\begin{document}

\maketitle

\begin{abstract}

We will strengthen the known upper and lower bounds on the delta-crossing number of knots in therms of the triple-crossing number. The latter bound turns out to be strong enough to obtain (unknown values of) triple-crossing numbers for a few knots. We also prove that we can always find at least one tangle from the set of four tangles, in any triple-crossing projections of any non-trivial knot or non-split link. In the last section, we enumerate and generate tables of minimal delta-diagrams for all prime knots up to the delta-crossing number equal to four. We also give a concise survey about known inequalities between integer-valued classical knot invariants.
\end{abstract}

\section{Introduction}

\noindent
It is known that any knot and any link has a delta-crossing diagram i.e. a diagram that can be decomposed into delta-crossing tangles \cite{NSS15}. The delta-crossing number of a knot $K$, denoted here by $c_{\Delta}(K)$, is defined as the least number of delta-crossings for any delta-crossing diagram of $K$. There are upper and lower bounds for the triple-crossing number \cite{NSS15}, in terms of the triple-crossing number $c_3$ and the canonical genus $g_c$, for knot $K$ we have: $2c_3(K)\geq c_{\Delta}(K) \text{ and }c_{\Delta}(K) \geq g_c(K).$
We will strengthen (because of the bound $c_3(K) \geq 2g_c(K)$ proved by the author in \cite{Jab20}) the lower bound.
\begin{theorem}\label{tmA}
	For any knot (or a link) $K$, we have $2c_{\Delta}(K) \geq c_3(K)$.
\end{theorem}

In Section \ref{s2} of this paper, we will also strengthen the above upper bound with additional assumptions as follows.
\begin{theorem}\label{tm1}
	For any knot $K$, let $t_1,t_2$ be the number of disjoint tangles $T_1,T_2$ (shown in Figure \ref{T1234}) respectively, embedded in the projection corresponding to a minimal triple-crossing diagram of $K$. Then we have $\;2c_3(K)-t_1-t_2\geq c_{\Delta}(K)$.
\end{theorem}

For a general context, we start with a survey (where the inequality in Theorem\;\ref{tmA} fits with the label $\star$), giving concise information about known inequalities between integer-valued classical knot invariants. We give the visual graph of relations, examples showing that some invariants are non-comparable in general, and show selected relations that are not resolved with a direction of possible inequality.
\par
We present the graphs of a selected invariant inequalities, for any knot $K\hookrightarrow \mathbb{S}^3$. Where, (in the first diagram) an arrow $\rightarrow$ means $\leq$ , (in the second diagram) an arrow with double arrowheads $\color{gray}\leftrightarrow$ means that there are examples of knots where we have $>$ relation and examples of knots where we have $<$ relation, the arrows $\color{red}\rightarrow$ with label $\color{red}\emph{conj.}$ are relation where we don't know examples where we might have $>$ relation.

\noindent
\begin{tikzcd}[column sep=1.1cm, row sep=1.1cm]
	{2\tau}\arrow[]{d}\arrow[]{dr}{\color{red} 24}&[-0.5cm]&[-0.5cm]&[-0.5cm]&[-0.5cm]2u_c\arrow[bend left=0]{dr}{\color{red} 31}&[-0.5cm]&[-0.5cm]&[-0.5cm]&[-0.5cm]&[-0.5cm]\\[-0.5cm]
	{2|\tau|}\arrow[pos=.3]{dr}{\color{red} 22}&[-0.5cm]{2\nu^+}\arrow[pos=.3]{d}{\color{red} 23}&[-0.5cm]{2cl_4}\arrow{r}{\color{red} 12}&[-0.5cm]{2u_s}\arrow[]{r}{\color{red} 25}&[-0.5cm]2u^*_r\arrow[bend left=2,pos=0.2]{r}{\color{red} 35}&[-0.5cm]2u\arrow[bend left=2]{r}{\color{red} 29}\arrow[bend left=40]{rr}{\color{red} 10}&[-0.5cm]2cl&[-0.5cm]tr\arrow[bend left=0]{drr}{\color{red} 26}\\[-0.5cm]
	{|s|}\arrow{r}{\color{red} 14}&2g_4 \arrow[bend left=2,pos=.8,swap]{ur}{\color{red} 11} \arrow{r}{\color{red} 19} \arrow[]{dr}{\color{red} 21} &[-0.5cm]2g_r \arrow[swap]{r}{\color{red} 18} \arrow[bend left=5,swap]{urr}{\color{red} 30} &[-0.5cm] ul_b \arrow{r}{\color{red} 17} &[-0.5cm] u_b\arrow[bend left=16,swap]{ur}{\color{red} 15}\arrow[swap]{r}{\color{red} 16}&[-0.5cm]2g \arrow[pos=.5]{d}{\color{red} 27} \arrow[swap]{r}{\color{red} 6}\arrow[bend right=18]{ur}{\color{red} 4}&[-0.5cm] 2g_f \arrow[]{r}{\color{red} 5}&[-0.5cm] 2g_c \arrow{r}{\color{red} 3}\arrow{u}[bend left=30,pos=.4,]{\color{red} 28} &[-0.5cm] c_3 \arrow{r}[bend left=30,swap]{\color{red} 2}\arrow{d}{\color{red} \star}&[-0.5cm]c\\[-0.5cm] 
	{|\sigma|}\arrow[bend right=20,swap]{rrrr}{\color{red} 34} \arrow[pos=.4]{ur}{\color{red} 13}&[-0.5cm]\varepsilon_{\Sigma}\arrow[swap]{r}{\color{red} 32}&[-0.5cm] g_{ds} \arrow[swap]{ur}{\color{red} 20}&[-0.5cm]Ord_v\arrow[swap, pos=.3]{u}{\color{red} 33}&[-0.5cm]sp\Delta_t\arrow[, pos=.6,swap]{ur}{\color{red} 8}&[-0.5cm]td \arrow[bend right=57]{urrrr}{\color{red} 1}&[-0.5cm]degP_z\arrow[pos=.6,swap]{ur}{\color{red} 7}&[-0.5cm]\nicefrac{spV_t}{2}\arrow[bend right=8]{ur}{\color{red} 9}&[-0.5cm]2c_{\Delta}\\
\end{tikzcd}

In the above graph, the invariants, are: $\color{blue} c$ is the crossing number,
$\color{blue} g$ is the (Seifert) three-genus,
$\color{blue} g_f$ in the free genus,
$\color{blue} g_c$ in the canonical genus,
$\color{blue} u$ is the unknotting number,
$\color{blue} u_b$ is the band-unknotting number,
$\color{blue} ul_b$ is the band-unlinking number,
$\color{blue} g_4$ is the slice genus,
$\color{blue} g_r$ is the ribbon slice genus,
$\color{blue} g_{ds}$ is the doubly slice genus, 
$\color{blue} \sigma$ is the knot signature,
$\color{blue} \tau$ is the Ozsvath-Szabo's Tau-Invariant,
$\color{blue} s$ is the Rasmussen’s s-invariant,
$\color{blue} sp\Delta_t$ is the span of the Alexander polynomial $\Delta(t)$,
$\color{blue} spV_t$ is the span of the Jones polynomial $V(t)$,
$\color{blue} degP_z$ is the $z$-degree of the HOMFLY-PT polynomial $P(v,z)$,
$\color{blue} cl_4$ is the 4D clasp number,
$\color{blue} cl$ is the clasp number,
$\color{blue} \nu^+$ is the Hom and Wu's invariant,
$\color{blue} u_s$ is the slicing number,
$\color{blue} td$ is the skein tree depth,
$\color{blue} tr$ is the trivializing number,
$\color{blue} u_c$ is the concordance unknotting number,
$\color{blue} u^*_r$ is the weak ribbon unknotting number,
$\color{blue} \varepsilon_{\Sigma}$ is the "double-cover" epsilon invariant,
$\color{blue} Ord_v$ is the torsion order.
\\

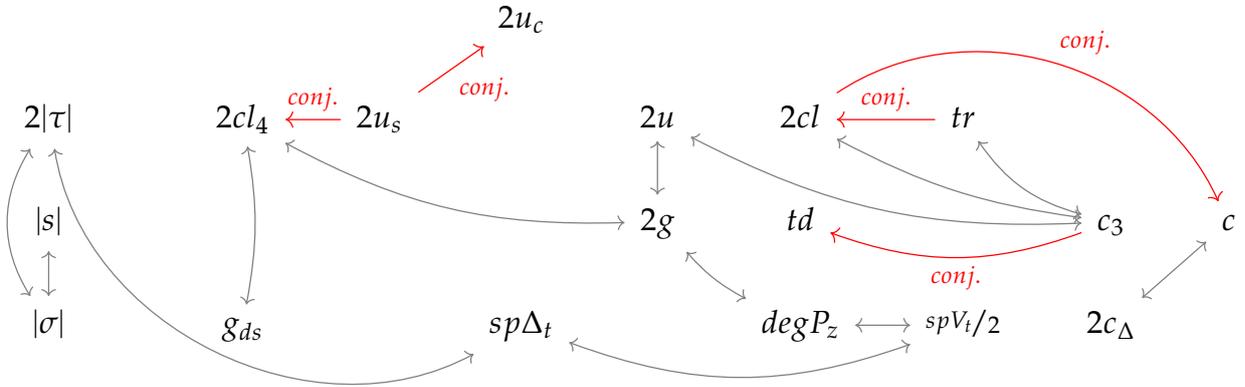
\noindent
\begin{tikzcd}[column sep=1.24cm, row sep=1.1cm]
	&[-0.5cm]&[-0.5cm]&[-0.5cm]&[-0.5cm]2u_c&[-0.5cm]&[-0.5cm]&[-0.5cm]&[-0.5cm]&[-0.5cm]\\[-0.5cm]
	{2|\tau|}\arrow[bend right=32,gray,<->]{dd}&[-0.5cm]&[-0.5cm]{2cl_4}&[-0.5cm]{2u_s}\arrow[red,swap]{l}{conj.}\arrow[red,swap]{ur}{conj.}&[-0.5cm]&[-0.5cm]2u\arrow[bend right=0,gray,<->]{d}\arrow[bend right=15,gray,<->]{drrr}&[-0.5cm]2cl\arrow[bend right=10,gray,<->]{drr}\arrow[bend left=50,red]{drrr}{conj.}&[-0.5cm]tr\arrow[bend right=20,gray,<->]{dr}\arrow[red,swap]{l}{conj.}\\[-0.5cm]
	{|s|}\arrow[bend right=0,gray,<->]{d}&&[-0.5cm]&[-0.5cm]&[-0.5cm]&[-0.5cm]2g \arrow[bend right=10,gray,<->]{dr}\arrow[bend left=15,gray,<->]{ulll}&[-0.5cm]td&[-0.5cm] &[-0.5cm] c_3 \arrow[bend left=20,red]{ll}{conj.}&[-0.5cm]c\\[-0.5cm]
	{|\sigma|} &[-0.5cm]&[-0.5cm] g_{ds}\arrow[bend right=10,gray,<->]{uu} &[-0.5cm]&[-0.5cm]sp\Delta_t\arrow[bend left=55,gray,<->]{uullll}\arrow[bend right=20,gray,<->]{rrr}&[-0.5cm]&[-0.5cm]degP_z\arrow[gray,<->]{r}&[-0.5cm]\nicefrac{spV_t}{2}&[-0.5cm]2c_{\Delta}\arrow[,gray,<->]{ur}\\
\end{tikzcd}

One find precise definitions of those invariants and inequalities as follows.\\ 
Relation {\color{red} 1} in \cite{Cro89},
{\color{red} 2, 9} in \cite{Ada13},
{\color{red} 3} in \cite{Jab20},
{\color{red} 5} in \cite{KobKob96},
{\color{red} 6} in \cite{Mor87},
{\color{red} 7} in \cite{Mor86},
{\color{red} 8} in \cite{Gil82},
{\color{red} 4, 11, 19, 25, 29, 35, 30} in \cite{Shi74},
{\color{red} 26} in \cite{Han14},
{\color{red} 13} in \cite{Mur65},
{\color{red} 14} in \cite{Ras10},
{\color{red} 17, 18, 33} in \cite{JMZ20},
{\color{red} 15, 16} in \cite{HNT90},
{\color{red} 20, 21} in \cite{KarSwe21},
{\color{red} 22} in \cite{OzsSza03},
{\color{red} 23, 24} in \cite{HomWu16},
{\color{red} 12, 31} in \cite{OweStr16},
{\color{red} 27} in \cite{SchTho89},
{\color{red} 10, 28} in \cite{Hetal11},
{\color{red} 32} in \cite{AGL17},
{\color{red} 34} in \cite{Fel16}.

The following invariants cannot be related to each other by an inequality for any knot.

\begin{enumerate}
	\item $u$ and $g$, for example $u(6a2)<g(6a2)$, $u(7a6)>g(7a6)$;
	\item $c_3$ and $2cl$, for example $c_3(9a40)<2cl(9a40)$, $c_3(5a1)>2cl(5a1)$;
	\item $c_3$ and $2u$, for example $c_3(9a40)<2u(9a40)$, $c_3(5a1)>2u(5a1)$;
	\item $|\sigma|$ and $|s|$, for example $|\sigma(10n13)|<|s(10n13)|$, $|\sigma(9n4)|>|s(9n4)|$;
	\item $|\sigma|$ and $2|\tau|$, for example $|\sigma(10n13)|<2|\tau(10n13)|$, $|\sigma(9n4)|>2|\tau(9n4)|$;
	
	\item $g_{ds}$ and $2cl_4$, for example $g_{ds}(7a6)<2cl_4(7a6)$, $g_{ds}(6a3)>2cl_4(6a3)$;
	
	\item $cl_4$ and $g$, for example $cl_4(6a2)<g(6a2)$, $cl_4(9a36)>g(9a36)$;
	
	\item $\nicefrac{spV_t}{2}$ and $sp\Delta_t$, for example $\nicefrac{spV_t(7a7)}{2}<sp\Delta_t(7a7)$, $\nicefrac{spV_t(7a4)}{2}>sp\Delta_t(7a4)$;
	\item $\nicefrac{spV_t}{2}$ and $degP_z$, for example $\nicefrac{spV_t(6a1)}{2}<degP_z(6a1)$, $\nicefrac{spV_t(6a3)}{2}>degP_z(6a3)$;
	\item $tr$ and $c_3$, for example $tr(9a36)<c_3(9a36)$, $tr(9n8)>c_3(9n8)$;
	\item $sp\Delta_t$ and $2|\tau|$, for example $sp\Delta_t(12n293)<2|\tau(12n293)|$, $sp\Delta_t(4a1)>2|\tau(4a1)|$;
	
	\item $2c_{\Delta}$ and $c$, for example $2c_{\Delta}(3a1)<c(3a1)$, $2c_{\Delta}(9a31)>c(9a31)$;
	\item $2g$ and $degP_z$, $2g(K15n14891)>degP_z(K15n14891)$, for $2g<degP_z$ see \cite{Mor86}.
	
\end{enumerate}

\par
In this paper, we also prove that we can always find at lest one of set of four tangles in any triple-crossing projections of any non-trivial knot. In \cite{NSS15} the authors list up all delta-crossing diagrams of prime knots with the delta-crossing number equal to $2$ and $3$ and found delta-crossing diagrams of prime knots with the delta-crossing number equal to $4$ for $52$ prime knots. In Section \ref{s3} of this paper, we enumerate and generate tables of minimal delta-diagrams for all prime knots up to the delta-crossing number equal to $4$.

\section{Definitions}\label{s1}

The \emph{projection} of a knot or a link $K\subset \mathbb{R}^3$ is its image under the standard projection $\pi:\mathbb{R}^3\to\mathbb{R}^2$ (or into a $2$-sphere) such that it has only a finite number of self-intersections, called \emph{multiple-points}, and in each multiple-point each pair of its strands are transverse. 
\par
If each multiple-point of a projection has multiplicity three then we call this projection a \emph{triple-point projection}. The \emph{triple-crossing} is a triple-point crossing with the strand labeled $T, M, B$, for top, middle and bottom.
\par
The \emph{triple-crossing diagram} is a triple-crossing projection such that each of its triple points is a triple-crossing, such that $\pi^{-1}$ of the strand labeled $T$ (in the neighborhood of that triple point) is on the top of the strand corresponding to the strand labeled $M$, and the latter strand is on the top of the strand corresponding to the strand labeled $B$ (see Figure \ref{r01}).

\begin{figure}[h!t]
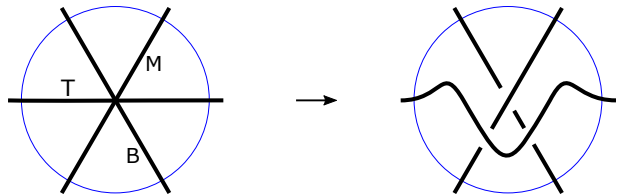

	\centering		
	\begin{lpic}[]{./M01(8cm)}
		
	\end{lpic}
	\caption{A deconstruction/construction of a triple-crossing.}
	\label{r01}
\end{figure}

A \emph{delta-crossing} is defined to be a tangle of three arcs with three double-crossings as in Figure \ref{deltacrossing}, which are appeared in a $\Delta$-move (or $\Delta$-unknotting operation \cite{MurNak89}). A \emph{delta-crossing diagram} is a double-crossing diagram that can be decomposed into delta-crossing tangles joined by simple arcs.

\begin{figure}[h!t]
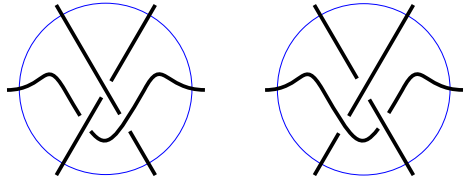

	\centering		
	\begin{lpic}[]{./deltacrossing(6cm)}
		
	\end{lpic}
	\caption{Two types of a unoriented delta-crossing.}
	\label{deltacrossing}
\end{figure}

The \emph{triple-crossing number} of a knot or link $K$, denoted $c_3(K)$, is the least number of triple-crossings for any triple-crossing diagram of $K$. The \emph{minimal triple-crossing diagram} of a knot $K$ is a triple-crossing diagram of $K$ that has exactly $c_3(K)$ triple-crossings. We define the \emph{delta-crossing number} of a knot $K$, denoted by $c_{\Delta}(K)$, as the least number of delta-crossings among all delta-crossing diagrams of $K$.
\par
A \emph{natural orientation} on a triple-crossing diagram or a delta-crossing diagram is an orientation of each component of that link, such that in each triple-crossing or a delta-crossing the (outer) strands are oriented in-out-in-out-in-out, as we encircle the crossing boundary. It is known (\cite{AHP19,NSS15}) that every orientation of the triple-crossing diagram or delta-crossing diagram obtained from an oriented knot is the natural orientation.

\section{Bound on delta-crossing number}\label{s2}

\begin{proof}[Proof of Theorem \ref{tmA}]
	Let $D^{\Delta}$ be a minimal delta-crossing diagram of a link $K$ with $n=c_{\Delta}$ number of delta-crossings. We can resolve $D^{\Delta}$ to a triple-crossing diagram with $2n$ triple-crossing by locally performing resolution of each delta-crossing into two triple-crossings as in Figure \ref{double_triple} (the other case of delta-crossing is analogous).

\begin{figure}[h!t]
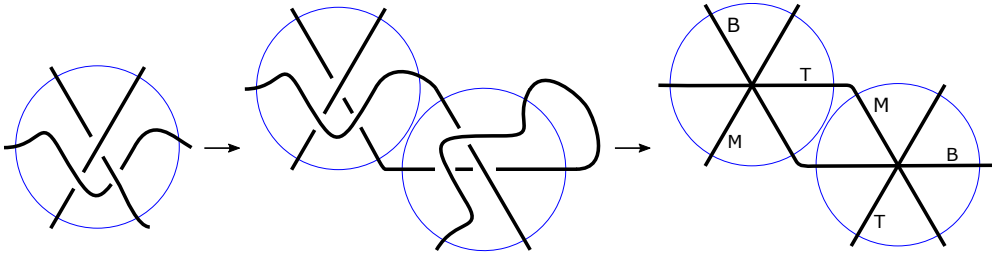

	\centering		
	\begin{lpic}[]{./double_triple(13cm)}
		
	\end{lpic}
	\caption{Transforming a delta-crossing triple-crossing into two triple-crossings.}
	\label{double_triple}
\end{figure}

\end{proof}

Because of equality of upper- and lower-bounds, an exact values of (up to now unknown) triple-crossing numbers for knots, such as the following can be obtained.

\begin{corollary}
	$$c_3(9a5) = c_3(9a13) = c_3(9a16) = 6.$$
\end{corollary}

\begin{proof}
	The above corollary follows from the facts that: results in Table\;\ref{table1} show that all these knots have the delta-crossing number equal to $3$; the inequality in Theorem\;\ref{tmA} tells us that all these knots have the triple-crossing number equal at most $6$; finally all these knots are not in the table of knots with the triple-crossing number equal at most $5$, obtained in \cite{Jab22}.
	
\end{proof}

Theorem \ref{tm1} follows from the following theorem.

\begin{theorem}\label{tm3}
	Let $D^t$ be a triple-crossing diagram of a link $L$ with $n$ triple-crossings and let $t_1,t_2$ be the number of disjoint tangles $T_1,T_2$ (shown in Figure \ref{T1234}) respectively embedded in the graph corresponding to $D^t$. Then $$\;2n-(t_1+t_2)\geq c_{\Delta}(L).$$
	
\end{theorem}

\begin{figure}[h!t]
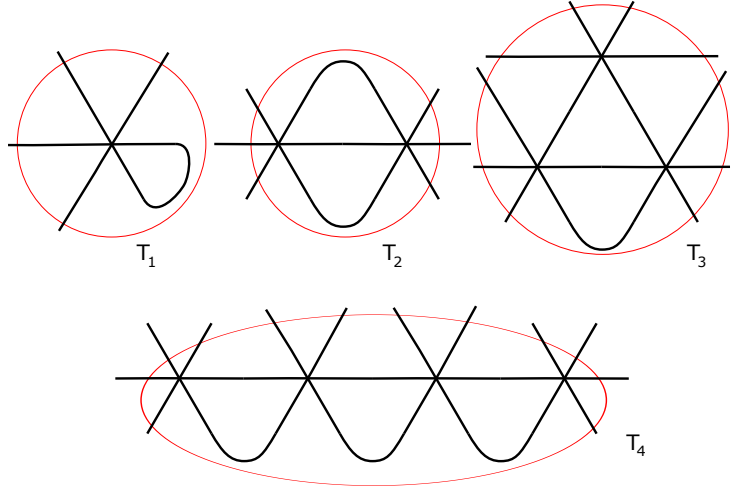

	\centering		
	\begin{lpic}[]{./T1234(9.5cm)}
		
	\end{lpic}
	\caption{Triple-point tangles $T_1,T_2,T_3, T_4$}
	\label{T1234}
\end{figure}

\begin{proof}
	
	From \cite{NSS15} we know that $D^t$ can be resolved to a delta-crossing diagram with $2n$ delta-crossing by locally performing resolution of each triple-crossing into two double-crossings as in Figure \ref{double} (the other case exchanging $T$ with $B$ is analogous).
	
\begin{figure}[h!t]
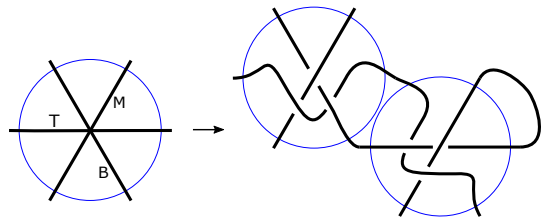

	\centering		
	\begin{lpic}[]{./double(7cm)}
		
	\end{lpic}
	\caption{Transforming a triple-crossing into two delta-crossings.}
	\label{double}
\end{figure}

To prove the theorem it is sufficient to show that any tangle from $T_1,T_2,T_3$ (with $r$ triple-crossings) in $D^t$ can be resolved to a delta-crossing tangle with at most $2r-1$ delta-crossings (we resolve the other triple-crossings not in any of the tangles as in Figure \ref{double}).
\par
We can transform a tangle $T_1$ to one delta-crossing tangle as in Figure \ref{1resolution}, the other case exchanging $T$ with $B$ is analogous. If in the tangle, the strand that is not forming the loop are not the middle strand, then the crossing can be eliminated by the standard Reidemeister moves.

\begin{figure}[h!t]
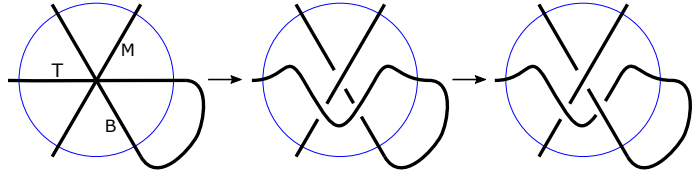

	\centering		
	\begin{lpic}[]{./1resolution(9cm)}
		
	\end{lpic}
	\caption{Transforming a triple-crossing with loop.}
	\label{1resolution}
\end{figure}

We can transform a tangle $T_2$ to at most three delta-crossings as in Figure \ref{2resolution} (the other case changing the labels $T$, $M$ and $B$ is analogous or resolves both triple-crossings to zero crossings).

\begin{figure}[h!t]
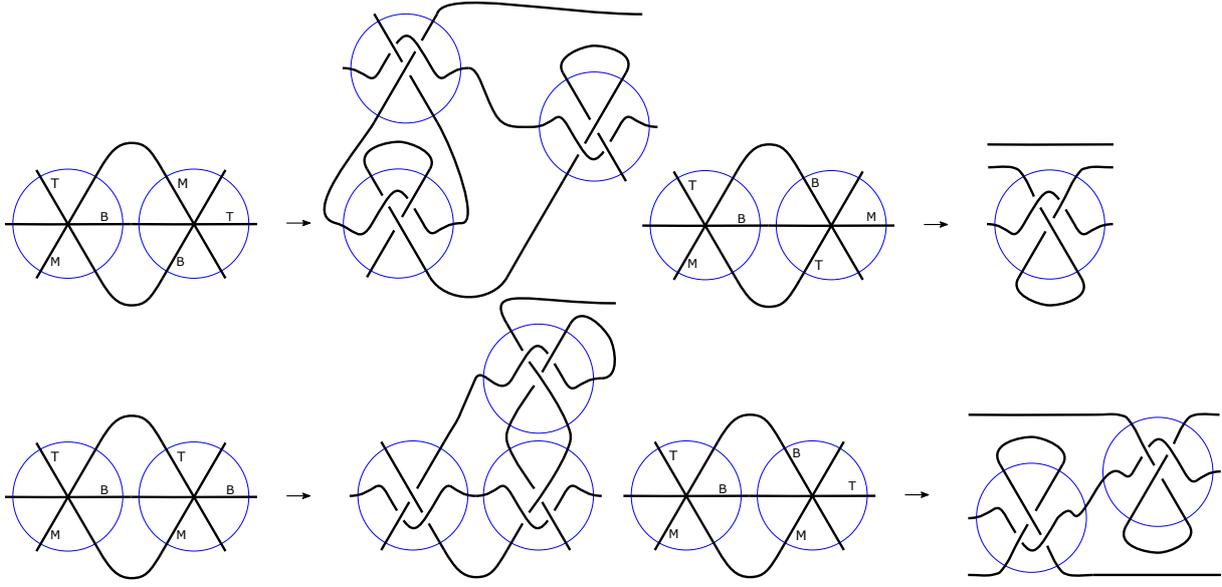

	\centering		
	\begin{lpic}[]{./2resolution(16cm)}
		
	\end{lpic}
	\caption{Transforming two triple-crossings.}
	\label{2resolution}
\end{figure}

\end{proof}

The tangles $T_1$ or $T_2$ do not always exist in any triple-crossing projection. However, we prove that we can always find at least one tangle from the following set of four tangles. 

\begin{theorem}\label{tm2}
	
	Let $D^t$ be a triple-crossing diagram of a non-trivial knot or non-split link $K$ and let $t_1,t_2,t_3,t_4$ be the number of disjoint tangles $T_1,T_2,T_3,T_4$ respectively (shown in Figure \ref{T1234}) embedded in the graph corresponding to $D^t$. Then $t_1+t_2+t_3+t_4\geq 1$.
	
\end{theorem}

\begin{proof}
	Denote by $n$ the number of crossings in $D^t$, and consider the projection of $D^t$ as a planar graph $G$.
	Let $f_i$ be the number of faces of $G$ with $i$ edges (including
	the outer region), from \cite{AHP19} we know that 
	$$2f_1+f_2=6+f_4+2f_5+3f_6+4f_7+\ldots .$$
	From the graph theory (using also the Euler formula) we have that in $G$ the number of vertices $|V(G)|=n$, the number of edges $|E(G)|=\frac{1}{2}(f_1+2f_2+3f_3+4f_4+\ldots)=3n$ and the number of faces $|F(G)|=f_1+f_2+f_3+f_4+\ldots=2n+2$.
	\par
	Assume the contrary that $t_1=t_2=t_3=t_4=0$. Then $f_1=0$, so we have 
	
	\begin{center}
		\begin{longtable}[ht]{rl}
			$f_2= $&$6+f_4+2f_5+3f_6+4f_7+\ldots \geq$ \\
			$\geq $&$6+f_4+f_5+f_6+f_7+\ldots=$\\
			$=$&$6+|F(G)|-f_2-f_3=2n+8-f_2-f_3.$
		\end{longtable}
	\end{center}
	
	\pagebreak
	
	Therefore $2f_2+f_3\geq 2n+8$. 
	\par
	Now we count the maximal number of bigons in $G$ avoiding having tangles $T_2,T_3,T_4$. To each $4$-gon there are at most $2$ adjacent bigons, to each $5$-gon there are at most $3$ adjacent bigons, to each $6$-gon there are at most $4$ adjacent bigons, ... , to each $m$-gon there are at most $\lfloor \frac{2m}{3} \rfloor$ adjacent bigons.
	\par
	
	Because each bigon is adjacent to two different $m$-gons (for $m>3$) we have
	
	\begin{center}
		\begin{longtable}[ht]{rl}
			$2f_2\leq $&$2f_4+3f_5+4f_6+4f_7+\ldots +\lfloor \frac{2m}{3} \rfloor f_m+\ldots\leq$ \\
			$\leq $&$2f_4+3f_5+4f_6+5f_7+\ldots +(m-2) f_m+\ldots=$\\
			$=$&$2|E(G)|-2|F(G)|-f_3=2n-4-f_3\leq 2f_2-12.$
		\end{longtable}
	\end{center}
	
	A contradiction, therefore $t_1+t_2+t_3+t_4\geq 1$.
	
\end{proof}

\begin{corollary}
	If a projection of a minimal triple-crossing diagram of non-trivial knot $K$ does not have embedded tangles neither $T_3$ nor $T_4$, then $2c_3(K)-1\geq c_{\Delta}(K)$.
\end{corollary}

\section{Knot tabulation}\label{s3}

We generate here the table of prime knot diagrams with the delta-crossing number up to four. First notice that we can obtain every delta-crossing diagram from a triple-point diagram because every triangle in a delta-crossing can be homotopically contracted to a point forming a triple-point. The tabulation of delta diagrams for just one delta-crossing is very easy so we may from now on assume that we have at least two delta-crossings.
\par
Then, we take the tables of all prime, connected, oriented, triple-point projections, such that each pair of shadows are neither isotopic nor is one of them isotopic to the mirror image of the other. This sets are taken from \cite{JabTro20} and called there $Tb_n$, where $n>1$ is the number of triple-points. From these sets we take only those that are projections of the knots, for $n=2, 3, 4$ we have $1, 5, 65$ projections respectively.
\par
Because every orientation of the delta-crossing diagram is the natural orientation, we can then resolve each triple-point to four possible delta-crossings $S, T, U, W$, as in Figure \ref{STUW}. This gives us $4^2+5\cdot 4^3+65\cdot 4^4=16976$ delta-diagrams to identify. We tabulate knots up to mirror images so we only need half the number of the diagrams. 

\begin{figure}[h!t]
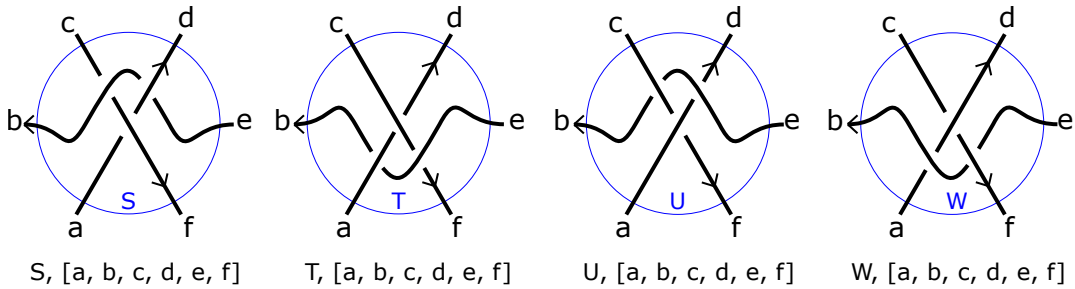

	\centering		
	\begin{lpic}[]{./STUW(14cm)}
		
	\end{lpic}
	\caption{Four types of delta-crossing.}
	\label{STUW}
\end{figure}

We identify each of the diagrams with the similar method as in our previous paper \cite{Jab22}, but now it turns out that we need much stronger invariants. To identify types of knots, we use mostly the HOMFLY-PT polynomial and the Khovanov homology  (calculated with \cite{SageMath}), we also use the knot Floer homology and Volume of the knot exterior (calculated with \cite{Snappy}) where they are needed. After that there were $12$ diagrams left for each of the trouble pairs ($11n76$ and $\overline{11n78}$), ($\overline{11n71}$, $11n75$) and ($11a44$, $11a47$) to identify, to do this we used diagrammatic manipulations (using \cite{KnotFolio} and verified using \cite{KLO22}). We also noticed that the quadruple $$\mathcal F(K)=(c(K),\text{HOMFLY-PT}(K), \text{Khovanov homology}(K), \text{knot Floer homology}(K))$$ does not determine the delta-crossing number $c_{\Delta}$, because for example\\ $\mathcal F(11a104)=\mathcal F(11a168)$ but $c_{\Delta}(11a104)=4$ and $c_{\Delta}(11a168)\not=4$.
\par
The number of knots and their names with a specific delta-crossing number is presented in Table \ref{table1}. We encode each delta-diagram as a \emph{dPD} code as a list in the format presented by an example in Figure \ref{K11n30}.

\begin{figure}[h!t]
	\centering		
	\begin{lpic}[]{./K11n30(12cm)}
		
	\end{lpic}
	\caption{A delta-diagram of a knot K11n30 and its dPD code.}
	\label{K11n30}
\end{figure}

Minimal diagrams of (unoriented) prime knots with delta-crossing number up to three are presented (in the form of its dPD codes) in Table \ref{table2}, the case for delta-diagrams with four delta-crossing can be found in the \texttt{delta4.txt} file in the \texttt{LaTeX} source file of this article’s \texttt{arXiv} preprint version. In that archive there are also text files of \texttt{sPD} codes of all the mentioned triple-crossing projections from $Tb_n$ sets.

\begin{footnotesize}
	
	\begin{center}

	\begin{longtable}[ht]{r||c|l}
			\caption{Enumeration of knots.\label{table1}}\\
	$\Delta$-crossings	& \# knots& name of knots\\
			\hline
\endfirsthead
\multicolumn{3}{c}
{\tablename\ \thetable\ -- \textit{Continued from previous page}} \\
$\Delta$-crossings	& \# knots& name of knots \\
\hline
\endhead
\hline \multicolumn{3}{r}{\textit{Continued on next page}} \\
\endfoot
\hline
\endlastfoot
		$1$&$1$&$3a1$.\\
		\hline
		$2$&$4$&$4a1$, $5a1$, $5a2$, $6a1$.\\
		\hline
		$3$&$21$&$6a2$,
		$6a3$,
		$7a1$,
		$7a2$,
		$7a3$,
		$7a4$,
		$7a5$,
		$7a6$,
		$7a7$,
		$8a1$,
		$8a2$,
		$8a3$,
		$8a4$,
		$8a6$,\\
		&&
		$8a7$,
		$8n1$,
		$8n2$,
		$8n3$,
		$9a5$,
		$9a13$,
		$9a16$.\\
		\hline
		$4$&$320$&
		 $8a5$,
		$8a8$,
		$8a9$,
		$8a10$,
		$8a11$,
		$8a12$,
		$8a13$,
		$8a14$,
		$8a15$,
		$8a16$,
		$8a17$,
		$8a18$,\\
		&&
		$9a1$,
		$9a2$,
		$9a3$,
		$9a4$,
		$9a6$,
		$9a7$,
		$9a8$,
		$9a9$,
		$9a10$,
		$9a11$,
		$9a12$,
		$9a14$,
		$9a15$,\\
		&&
		$9a17$,
		$9a18$,
		$9a19$,
		$9a20$,
		$9a21$,
		$9a22$,
		$9a23$,
		$9a24$,
		$9a25$,
		$9a26$,
		$9a27$,\\
		&&
		$9a28$,
		$9a29$,
		$9a30$,
		$9a32$,
		$9a33$,
		$9a34$,
		$9a35$,
		$9a36$,
		$9a38$,
		$9a39$,
		$9a40$,\\
		&&
		$9a41$,
		$9n1$,
		$9n2$,
		$9n3$,
		$9n4$,
		$9n5$,
		$9n6$,
		$9n7$,
		$9n8$,
		$10a3$,
		$10a4$,
		$10a5$,\\
		&&
		$10a6$,
		$10a7$,
		$10a8$,
		$10a9$,
		$10a10$,
		$10a11$,
		$10a12$,
		$10a13$,
		$10a14$,
		$10a15$,\\
		&&
		$10a16$,
		$10a17$,
		$10a18$,
		$10a19$,
		$10a20$,
		$10a21$,
		$10a25$,
		$10a26$,
		$10a28$,\\
		&&
		$10a29$,
		$10a30$,
		$10a31$,
		$10a32$,
		$10a33$,
		$10a34$,
		$10a36$,
		$10a37$,
		$10a38$,\\
		&&
		$10a39$,
		$10a40$,
		$10a41$,
		$10a42$,
		$10a43$,
		$10a44$,
		$10a45$,
		$10a47$,
		$10a48$,\\
		&&
		$10a49$,
		$10a50$,
		$10a51$,
		$10a52$,
		$10a55$,
		$10a56$,
		$10a57$,
		$10a58$,
		$10a63$,\\
				&&
		$10a64$,
		$10a66$,
		$10a67$,
		$10a68$,
		$10a69$,
		$10a72$,
		$10a78$,
		$10a79$,
		$10a80$,\\
		&&
		$10a85$,
		$10a88$,
		$10a93$,
		$10a94$,
		$10a99$,
		$10a102$,
		$10a103$,
		$10a106$,
		$10a107$,\\
		&&
		$10a108$,
		$10a109$,
		$10a118$,
		$10a121$,
		$10n1$,
		$10n2$,
		$10n4$,
		$10n5$
		$10n6$,\\
		&&
		$10n7$,
		$10n8$,
		$10n9$,
		$10n10$,
		$10n11$,
		$10n12$,
		$10n13$,
		$10n15$,
		$10n16$,
		$10n17$,\\
		&&
		$10n18$,
		$10n19$,
		$10n20$,
		$10n21$,
		$10n22$,
		$10n23$,
		$10n25$,
		$10n26$,
		$10n27$,\\
		&&
		$10n28$,
		$10n29$,
		$10n30$,
		$10n31$,
		$10n32$,
		$10n33$,
		$10n34$,
		$10n35$,
		$10n36$,\\
		&&
		$10n37$,
		$10n38$,
		$10n42$,
		$11a11$,
		$11a14$,
		$11a18$,
		$11a22$,
		$11a23$,
		$11a27$,\\
		&&
		$11a30$,
		$11a32$,
		$11a35$,
		$11a36$,
		$11a40$,
		$11a41$,
		$11a43$,
		$11a44$,
		$11a46$,\\
		&&
		$11a47$,
		$11a50$,
		$11a52$,
		$11a77$,
		$11a83$,
		$11a85$,
		$11a91$,
		$11a94$,
		$11a95$,\\
		&&
		$11a100$,
		$11a101$,
		$11a104$,
		$11a106$,
		$11a107$,
		$11a109$,
		$11a111$,
		$11a114$,\\
		&&
		$11a117$,
		$11a123$,
		$11a134$,
		$11a135$,
		$11a136$,
		$11a138$,
		$11a148$,
		$11a155$,\\
		&&
		$11a175$,
		$11a177$,
		$11a178$,
		$11a186$,
		$11a191$,
		$11a192$,
		$11a197$,
		$11a198$,\\
		&&
		$11a200$,
		$11a212$,
		$11a327$,
		$11a329$,
		$11n1$,
		$11n2$,
		$11n16$,
		$11n21$,
		$11n22$,\\
		&&
		$11n23$,
		$11n24$,
		$11n30$,
		$11n54$,
		$11n55$,
		$11n56$,
		$11n61$,
		$11n69$,
		$11n71$,\\
		&& 
		$11n72$,
		$11n73$,
		$11n74$,
		$11n75$,
		$11n76$,
		$11n77$,
		$11n78$,
		$11n82$,
		$11n85$,\\
		&&
		$11n86$,
		$11n87$,
		$11n90$,
		$11n92$,
		$11n93$,
		$11n94$,
		$11n95$,
		$11n96$,
		$11n105$,\\
		&&
		$11n106$,
		$11n107$,
		$11n118$,
		$11n119$,
		$11n126$,
		$11n136$,
		$11n153$,
		$11n156$,\\
		&&
		$11n162$,
		$11n169$,
		$12a58$,
		$12a99$,
		$12a104$,
		$12a119$,
		$12a268$,
		$12a273$,
		$12a281$,\\
		&&
		$12a295$,
		$12a313$,
		$12a323$,
		$12a327$,
		$12a345$,
		$12a353$,
		$12a426$,
		$12a435$,\\
		&&
		$12a499$,
		$12a510$,
		$12a514$,
		$12a561$,
		$12a615$,
		$12a628$,
		$12a629$,
		$12a631$,\\
		&&
		$12a633$,
		$12a653$,
		$12a656$,
		$12a868$,
		$12a875$,
		$12a960$,
		$12a1097$,
		$12a1188$,\\
		&&
		$12a1189$,
		$12a1251$,
		$12n41$,
		$12n77$
		$12n177$,
		$12n188$,
		$12n245$,
		$12n289$,\\
		&&
		$12n308$,
		$12n326$,
		$12n327$,
		$12n328$,
		$12n341$,
		$12n379$,
		$12n380$,
		$12n406$,\\
		&&
		$12n416$,
		$12n417$,
		$12n425$,
		$12n426$,
		$12n477$,
		$12n503$,
		$12n508$,
		$12n518$,\\
		&&
		$12n538$,
		$12n549$,
		$12n591$,
		$12n592$,
		$12n600$,
		$12n609$,
		$12n703$,
		$12n706$.\\

	\end{longtable}
\end{center}

\end{footnotesize}

\pagebreak

\begin{footnotesize}
	\renewcommand{\arraystretch}{1.25}
	
	\begin{center}
		\begin{longtable}{r||l}
			\caption{Minimal diagrams of prime knots with delta-crossing number up to three\label{table2}}\\
			knot & $dPD$ code of a minimal knot diagram \\
			\hline
			\endfirsthead
			\multicolumn{2}{c}
			{\tablename\ \thetable\ -- \textit{Continued from previous page}} \\
			knot & $dPD$ code of a minimal knot diagram \\
			\hline
			\endhead
			\hline \multicolumn{2}{r}{\textit{Continued on next page}} \\
			\endfoot
			\hline
			\endlastfoot
			3a1& [W, [1, 3, 3, 2, 2, 1]]\\
4a1& [S, [5, 2, 4, 6, 1, 5], W, [2, 1, 3, 3, 6, 4]]\\
5a1& [W, [5, 2, 4, 6, 1, 5], W, [2, 1, 3, 3, 6, 4]]\\
5a2& [W, [5, 2, 4, 6, 1, 5], U, [2, 1, 3, 3, 6, 4]]\\
6a1& [W, [5, 2, 4, 6, 1, 5], S, [2, 1, 3, 3, 6, 4]]\\
6a2 & [S, [7, 2, 6, 8, 1, 7], W, [2, 1, 3, 3, 9, 4], U, [4, 9, 5, 5, 8, 6]]\\
6a3 & [S, [7, 2, 6, 8, 1, 7], W, [2, 1, 3, 3, 9, 4], W, [4, 9, 5, 5, 8, 6]]\\
7a1 & [T, [4, 2, 5, 5, 1, 6], T, [7, 3, 8, 8, 2, 9], U, [9, 4, 6, 1, 3, 7]]\\
7a2 & [U, [7, 2, 6, 8, 1, 7], T, [2, 1, 3, 3, 9, 4], U, [4, 9, 5, 5, 8, 6]]\\
7a3 & [W, [7, 2, 6, 8, 1, 7], W, [2, 1, 3, 3, 9, 4], U, [4, 9, 5, 5, 8, 6]]\\
7a4 & [W, [7, 2, 6, 8, 1, 7], W, [2, 1, 3, 3, 9, 4], W, [4, 9, 5, 5, 8, 6]]\\
7a5 & [W, [4, 2, 5, 5, 1, 6], U, [7, 3, 8, 8, 2, 9], W, [9, 4, 6, 1, 3, 7]]\\
7a6 & [W, [4, 2, 5, 5, 1, 6], W, [7, 3, 8, 8, 2, 9], W, [9, 4, 6, 1, 3, 7]]\\
7a7 & [U, [4, 2, 5, 5, 1, 6], U, [7, 3, 8, 8, 2, 9], W, [9, 4, 6, 1, 3, 7]]\\
8a1 & [T, [4, 2, 5, 5, 1, 6], U, [7, 3, 8, 8, 2, 9], U, [9, 4, 6, 1, 3, 7]]\\
8a2 & [U, [7, 2, 6, 8, 1, 7], U, [2, 1, 3, 3, 9, 4], U, [4, 9, 5, 5, 8, 6]]\\
8a3 & [T, [7, 2, 6, 8, 1, 7], S, [2, 1, 3, 3, 9, 4], U, [4, 9, 5, 5, 8, 6]]\\
8a4 & [T, [7, 2, 6, 8, 1, 7], T, [2, 1, 3, 3, 9, 4], U, [4, 9, 5, 5, 8, 6]]\\
8a6 & [S, [4, 2, 5, 5, 1, 6], U, [7, 3, 8, 8, 2, 9], T, [9, 4, 6, 1, 3, 7]]\\
8a7 & [T, [4, 2, 5, 5, 1, 6], U, [7, 3, 8, 8, 2, 9], T, [9, 4, 6, 1, 3, 7]]\\
8n1 & [S, [7, 2, 6, 8, 1, 7], S, [2, 1, 3, 3, 9, 4], U, [4, 9, 5, 5, 8, 6]]\\
8n2 & [S, [7, 2, 6, 8, 1, 7], U, [2, 1, 3, 3, 9, 4], U, [4, 9, 5, 5, 8, 6]]\\
8n3 & [W, [7, 2, 6, 8, 1, 7], U, [2, 1, 3, 3, 9, 4], U, [4, 9, 5, 5, 8, 6]]\\
9a5 & [T, [7, 2, 6, 8, 1, 7], U, [2, 1, 3, 3, 9, 4], U, [4, 9, 5, 5, 8, 6]]\\
9a13 & [U, [4, 2, 5, 5, 1, 6], U, [7, 3, 8, 8, 2, 9], T, [9, 4, 6, 1, 3, 7]]\\
9a16 & [U, [4, 2, 5, 5, 1, 6], U, [7, 3, 8, 8, 2, 9], U, [9, 4, 6, 1, 3, 7]]\\
			\hline
		\end{longtable}
	\end{center}
\end{footnotesize}


\begin{thebibliography}{99}
	
	
	\bibitem{AGL17} P. Aceto, M. Golla, and K. Larson. Embedding 3-manifolds in spin 4-manifolds, \emph{Journal of Topology} 10 (2017), 301--323.
	
	\bibitem{Ada13} C. Adams, Triple crossing number of knots and links, \emph{J. Knot Theory Ramifications} 22 (2013), 1350006.
	
	\bibitem{AHP19} C. Adams, J. Hoste and M. Palmer, Triple-crossing number and moves on triple-crossing link diagram, \emph{J. Knot Theory Ramifications} 28 (2019), 1940001.
	
	\bibitem{Cro89} P. Cromwell, Homogeneous links, \emph{Journal of the London Mathematical Society} 2 (1989) 535--552.
	
	\bibitem{Snappy} M. Culler, N.M. Dunfield, M. Goerner and J.R. Weeks, \emph{Snap{P}y, a computer program for studying the geometry and topology of $3$-manifolds}, (Version 3.0.3) available at \url{http://snappy.computop.org} (2022).
	
	\bibitem{Fel16} P. Feller, The degree of the Alexander polynomial is an upper bound for the topological slice genus, \emph{Geometry \& Topology} 20 (2016), 1763--1771.
	
	\bibitem{Gil82} C.A. Giller, A family of links and the Conway calculus, \emph{Transactions of the American Mathematical Society} 270 (1982) 75--109.
	
	\bibitem{Han14} R. Hanaki, Trivializing number of knots, \emph{Journal of the Mathematical Society of Japan} 66 (2014), 435--447.
	
	\bibitem{Hetal11} A. Henrich, N. MacNaughton, S. Narayan, O. Pechenik and J. Townsend, Classical and virtual pseudodiagram theory and new bounds on unknotting numbers and genus. \emph{Journal of Knot Theory and Its Ramifications} 20 (2011), 625--650.
	
	\bibitem{HomWu16} J. Hom and Z. Wu, Four-ball genus bounds and a refinement of the Ozsváth-Szabó tau-invariant, \emph{Journal of Symplectic Geometry} 14 (2016), 305--323.
	
	\bibitem{HNT90} J. Hoste, Y. Nakanishi, and K. Taniyama, Unknotting operations involving trivial tangles, \emph{Osaka Journal of Mathematics} 27 (1990) 555--566.
	
	\bibitem{Jab20} M. Jab\l onowski, Triple-crossing number, the genus of a knot or link and torus knots, \emph{Topology and its Applications} 285 (2020), 107389.
	
	\bibitem{Jab22} M. Jab\l onowski, Tabulation of knots up to five triple-crossings and moves between oriented diagrams, to appear in \emph{Tokyo Journal of Mathematics} (2022).
	
	\bibitem{JabTro20} M. Jab\l onowski and \L. Trojanowski, Triple-crossing projections, moves on knots and links, and their minimal diagrams, \emph{J. Knot Theory Ramifications} 29 (2020), 2050015.
	
	\bibitem{JMZ20} A. Juhász, M. Miller, and I. Zemke, Knot cobordisms, bridge index, and torsion in Floer homology,  \emph{Journal of Topology} 13 (2020), 1701--1724.
	
	\bibitem{KarSwe21} L.P. Karageorghis and F. Swenton, Determining the doubly slice genera of prime knots with up to 12 crossings, \emph{J. Knot Theory Ramifications} 30 (2021), 2150057.
	
	\bibitem{KobKob96} M.Kobayashi and T. Kobayashi. On canonical genus and free genus of knot. \emph{J. Knot Theory Ramifications} 5 (1996) 77–-85.
	
	\bibitem{KnotFolio} K. Miller, \emph{KnotFolio}, \url{https://kmill.github.io/knotfolio/} (2022).
	
	\bibitem{Mor87} Y. Moriah, On the free genus of knots, \emph{Proceedings of the American Mathematical Society} (1987), 373--379.
	
	\bibitem{Mor86} H.R. Morton, Seifert circles and knot polynomials, \emph{Mathematical Proceedings of the Cambridge Philosophical Society} 99 (1986), 247-–260.
	
	\bibitem{MurNak89} H. Murakami, and Y. Nakanishi, On a certain move generating link-homology, \emph{Mathematische Annalen} 284.1 (1989), 75--89.
	
	\bibitem{Mur65} K. Murasugi, On a certain numerical invariant of link types, \emph{Transactions of the American Mathematical Society} 117 (1965), 387--422.
	
	\bibitem{NSS15} Y. Nakanishi, Y.  Sakamoto and S. Satoh, Delta-crossing number for knots, \emph{Topology and its Applications} 196 (2015), 771--776.
	
	\bibitem{OweStr16} B. Owens and S. Strle. Immersed disks, slicing numbers and concordance unknotting numbers, \emph{Communications in Analysis and Geometry} 24 (2016), 1107--1138.
	
	\bibitem{OzsSza03} P. Ozsváth and Z. Szabó, Knot Floer homology and the four-ball genus, \emph{Geom. Topol.} 7 (2003), 615--639.
	
	\bibitem{Ras10} J. Rasmussen, Khovanov homology and the slice genus, \emph{Inventiones mathematicae} 182 (2010), 419--447.
	
	\bibitem{SageMath} SageMath, the Sage Mathematics Software System (Version 9.3),
	The Sage Developers, \url{https://www.sagemath.org} (2022).
	
	\bibitem{SchTho89} M. Scharlemann and A. Thompson, Link genus and the Conway moves, \emph{Comment. Math. Helv.} 64 (1989), 527--535
	
	\bibitem{Shi74} T. Shibuya, Some relations among various numerical invariants for links, \emph{Osaka Journal of Mathematics} 11 (1974), 313--322.
	
	\bibitem{KLO22} F.J. Swenton, Kirby calculator (v0.973a) \url{http://community.middlebury.edu/~mathanimations/klo/} (2022)
	
\end{thebibliography}
\end{document}